\documentclass[12pt]{amsart}
\usepackage{amsmath,amssymb,amsthm,amsfonts,mathrsfs,amsopn}
\usepackage[all]{xy}
\usepackage{dsfont}
\usepackage{hyperref}
\usepackage{color}
\usepackage{esint}
\usepackage{mathtools}
\mathtoolsset{showonlyrefs}
\usepackage{slashed}

\usepackage{tikz-cd}

\usepackage{textgreek}

\usepackage[margin=1.1in]{geometry}

\usepackage[obeyDraft]{todonotes}

\newcommand{\parenthesis}[1]{\left(#1\right)} % round bracket
\newcommand{\braces}[1]{\left\{#1\right\}} % curly bracket

\newcommand{\average}[1]{\left<#1\right>_{_{\lambda}}} % lambda-average
\newcommand{\oscil}[1]{\left[#1\right]_{_{\lambda}}} % \lambda-oscillation part

\newcommand{\R}{\mathbb{R}}

\newcommand{\cF}{\mathcal{F}}
\newcommand{\cE}{\mathcal{E}}

\newcommand{\cP}{\mathcal{P}}
\newcommand{\cS}{\mathcal{S}}

\newcommand{\ba}{\boldsymbol{\alpha}}

\newcommand{\dd}{\mathop{}\!\mathrm{d}}
\newcommand{\eps}{\varepsilon}
\newcommand{\p}{\partial}

\newcommand{\RH}{\varrho}

\newcommand{\dx}{\dd{x}}
\newcommand{\dy}{\dd{y}}

\DeclareMathOperator{\dist}{dist}

\newtheorem{thm}{Theorem}[section]
\newtheorem{dfn}[thm]{Definition}
\newtheorem{lemma}[thm]{Lemma}
\newtheorem{prop}[thm]{Proposition}

\newtheorem{rmk}[thm]{Remark}

\numberwithin{equation}{section}

\title[Global bifurcation diagram]{On the global bifurcation diagram of the equation $-\Delta u=\mu|x|^{2\alpha}e^u$ in dimension two}

\author[D. Bartolucci, A. Jevnikar, R. Wu]{Daniele Bartolucci, Aleks Jevnikar, Ruijun Wu}

\address{Daniele Bartolucci, Department of Mathematics, University of Rome {\it "Tor Vergata"}, Via della ricerca scientifica n.1, 00133 Roma, Italy. }
\email{bartoluc@mat.uniroma2.it}

\address{Aleks Jevnikar, Department of Mathematics, Computer Science and Physics, University of Udine, Via delle Scienze 206, 33100 Udine, Italy.}
\email{aleks.jevnikar@uniud.it}

\address{Ruijun Wu, School of Mathematics and Statistics, Beijing Institute of Technology, Zhongguancun South Street No. 5, Haidian District, Beijing, P.R. China.}
\email{ruijun.wu@bit.edu.cn}

\begin{document}


\thanks{The research of the first author is partially supported by the MIUR Excellence Department Project\\ MatMod@TOV
awarded to the Department of Mathematics, University of Rome Tor Vergata.}

\begin{abstract}
The aim of this note is to present the first qualitative global bifurcation diagram of the equation $-\Delta u=\mu|x|^{2\alpha}e^u$. To this end, we introduce the notion of domains of first/second kind for singular mean field equations and base our approach on a suitable spectral analysis. In particular, we treat also non-radial solutions and non-symmetric domains and show that the shape of the branch of solutions still resembles the well-known one of the model regular radial case on the disk. Some work is devoted also to the asymptotic profile for $\mu\to-\infty$.
\end{abstract}

\maketitle


\

\section{Introduction}

Let~$\Omega\subset \R^2$ be a smooth bounded simply connected domain. We consider, for $\alpha>0$ and $\mu\in\R$, the equation
$$
-\Delta v= \mu |x|^{2\alpha} e^v \quad \mbox{ in } \Omega
$$
with Dirichlet boundary conditions, or, more in general
\begin{align}\tag{$L_\mu$} \label{eq:L-mu}
 \begin{cases}
  -\Delta v= \mu h(x) e^v & \mbox{ in } \Omega \\
  v=0 & \mbox{ on } \p\Omega,
 \end{cases}
\end{align}
with a singular weight
$$
	h(x)=\exp\left(-4\pi\sum_{j=1}^n\alpha_j G_{p_j}(x)  \right),
$$
where $\{p_1,\dots,p_n\}\subset\Omega$, $\alpha_j>0$ and $G_p$, $p\in\Omega$, is the Green function
$$
	 \begin{cases}
  -\Delta G_p(x)= \delta_p & \mbox{ in } \Omega \\
  G_p(x)=0 & \mbox{ on } \p\Omega.
 \end{cases}
$$
In particular, we have
$$
	h>0 \mbox{ in } \Omega\setminus\{p_1,\dots,p_n\}, \quad h(x)\approx |x-p_j|^{2\alpha_j} \mbox{ near } p_j.
$$
The Gelfand equation \eqref{eq:L-mu} and its mean field formulation \eqref{eq:MP-lambda} have attracted a
lot of attention due to their relevance in geometry and mathematical physics and there are by now many available results in the literature; we refer the interested readers for example to the introduction of \cite{BJLY} and the references quoted therein. However, at least to our knowledge, with the unique exception of the well-known case where $\Omega$ is a disk
with a singularity sitting in its center (see for example \cite{bartolucci2009uniqueness}),
there are no results about the global bifurcation diagram of \eqref{eq:L-mu}. The aim of this note is to fill this gap by obtaining a neat qualitative description of the shape of the Rabinowitz continuum emanating from the trivial solution with $\mu=0$, see \cite{rab}, for a class of general domains, even in a non-radial setting. In particular, the bifurcation diagram presents a bending point after which the first
eigenvalue of the linearized equation for \eqref{eq:L-mu} is negative and it is thus difficult to catch any monotonicity property of its solutions.

\

The approach is based on a suitable spectral analysis related to the singular mean field equation
\begin{align}\tag{$MP_\lambda$} \label{eq:MP-lambda}
 \begin{cases}
  -\Delta\psi_\lambda= \dfrac{h(x) e^{\lambda\psi_\lambda}}{\int_{\Omega} h(x) e^{\lambda\psi_{\lambda}}\dx }, & \mbox{ in } \Omega \\
  \psi_{\lambda}=0, & \mbox{ on } \p\Omega,
 \end{cases}
\end{align}
for $\lambda\in\R$, in the spirit of \cite{bartolucci2019global, bartolucci2019ontheglobal}. It is well-known that, due to the Moser-Trudinger inequality, for any $\lambda<8\pi$ the equation~\eqref{eq:MP-lambda} admits a solution, which is unique for $\Omega$ simply connected \cite{bartolucci2009uniqueness}, see also Remark \ref{rem:unique}. Let us also point out that if some $\alpha_j<0$, then the existence/uniqueness threshold is not $8\pi$ anymore; since we do not have a full understanding of this case we postpone its study to a future work. Finally, the critical/super critical case $\lambda\geq8\pi$ is subtler and existence/uniqueness is not granted in general.

\medskip

Following the pioneering work about the regular case \cite{CLMP}, we introduce the following notion for our singular setting. We first collect the set of singularities $p_j$ of orders $\alpha_j$ in the formal sum
$$
	\ba=\sum_{j=1}^n \alpha_j p_j,
$$
and denote by $(\Omega,\ba)$ the domain with that set of singularities.

\begin{dfn}
$(\Omega,\ba)$ is said to be a domain of first kind if for~$\lambda=8\pi$,~\eqref{eq:MP-lambda} admits no solution and of second kind otherwise.
\end{dfn}

\begin{rmk}
The existence/non existence problem for $\lambda=8\pi$ has been discussed in details in \cite{BL}. For example, consider a symmetric domain $\Omega$ with a singularity $p_1$ sitting in its center. Then, $(MP_{8\pi})$ has a solution for any $\alpha_1>0$. On the contrary, take a disk with a singularity not centered at the origin. Then, there exist $\alpha_-\leq\alpha_+$ such that $(MP_{8\pi})$ admits a solution for any $\alpha_1>\alpha_+$ and no solution for any $\alpha_1\leq\alpha_-$. The authors also show that for general domains with multiple singularities, under suitable assumptions, for small $\alpha_j$ the solvability of $(MP_{8\pi})$ is equivalent to the solvability of the corresponding regular problem, for which we refer to \cite{BDM,BL-ann,CCL}.

In particular, the class of first kind domains contains also non-symmetric configurations $(\Omega,\ba)$ for which a branch of non-radial solutions exists.
\end{rmk}

\

The idea is to parametrize the bifurcation diagram of \eqref{eq:L-mu} by using the (Dirichlet) energy of the unique, for $\lambda<8\pi$, solution $\psi_\lambda$ of \eqref{eq:MP-lambda}, that is
\begin{align}
 E_\lambda\equiv E(\psi_\lambda)
 =\frac{1}{2}\int_\Omega |\nabla\psi_\lambda|^2\dx.
\end{align}
It turns out that the function~$\lambda\mapsto E_\lambda\eqqcolon E(\lambda)$ enjoys some nice properties as we show in the next result.
\begin{thm} \label{thm1}
Let $(\Omega,\ba)$ be a domain of first kind. The function~$E\colon (-\infty, 8\pi) \to (0,+\infty)$ is strictly increasing and bijective, see Figure~\ref{pic:lambda-E}. In particular,
$$
	\lim_{\lambda\to-\infty} E(\lambda)=0, \quad \lim_{\lambda\to8\pi^-} E(\lambda)=+\infty.
$$
\end{thm}

\begin{figure}[h]
 \centering
 \begin{tikzpicture}
 \draw[gray, very thin,->] (0,-1) -- (0,5) node[left]{$E$};
 \draw[gray, very thin,->] (-3,0) -- (3,0) node[right]{$\lambda$};
 \draw[dashed]  (2,-1) - - (2, 5) node[pos=0.13, left]{$8\pi$};
 %\draw[dashed]  (-3,0.5)-- (3,0.5) node[pos=0.55, above]{$\underline{E}$};
 \draw[thick] plot [smooth] coordinates{(-3,0.05) (0,0.4) (1.2,1.5)(1.9,5)};
 \draw (0,0.6)  node[left]{$E_0$};
 \draw (0, -0.2) node[left]{$0$};
\end{tikzpicture}\\
 \caption{The energy of the solutions.}
 \label{pic:lambda-E}
\end{figure}
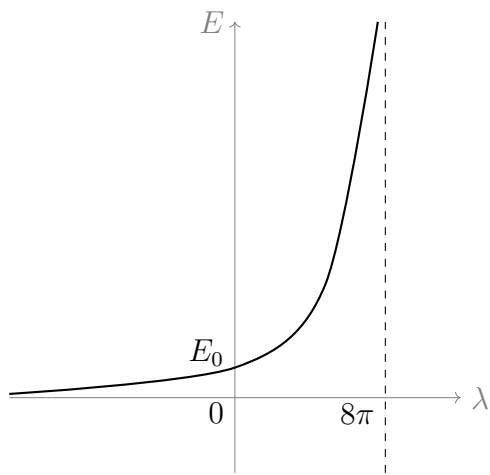

\

We evaluate the sign of the derivative of $E_\lambda$ by introducing a suitable spectral theory and by looking in particular at the Fourier modes of the linearized operator of \eqref{eq:MP-lambda}. The asymptotic for $\lambda\to8\pi^-$ is a simple consequence of the choice of the domain $(\Omega,\ba)$. On the other hand, some work is needed to prove the asymptotic for $\lambda\to-\infty$. Its regular counterpart was treated in 
\cite{bartolucci2019ontheglobal} based on the existence theory of free energy functionals in \cite{caglioti1995special}; however, due to the singular weight we need to attack the problem with a different strategy by using an approximation argument jointly with suitable test functions.

\

Now, for a solution~$(\lambda,\psi_\lambda)$ of~\eqref{eq:MP-lambda}, we have that
\begin{align}\label{eq:from lambda to mu}
 \mu_\lambda=\frac{\lambda}{\int_\Omega h e^{\lambda\psi_\lambda}\dx }, \quad
 u_\lambda=\lambda\psi_\lambda
\end{align}
is a solution of~\eqref{eq:L-mu}. Our aim is then to parametrize the bifurcation diagram by the energy expressed in terms of $\mu$, that is, for $\mu, \lambda\neq0$,
\begin{align}
 E\mu\equiv E(v_\mu)
 \coloneqq &\frac{1}{2}\int_\Omega |\nabla\psi_\lambda|^2\dx
 =\frac{1}{2\lambda^2}\int_\Omega |\nabla v_\mu|^2\dx \\
 =&\frac{1}{2\mu^2}\frac{1}{\parenthesis{\int_\Omega h e^{v_\mu}dy}^2}\int_\Omega (-\Delta v_\mu)v_\mu\dx \\
 =&\frac{1}{2\mu}\frac{1}{\int_\Omega he^{v_\mu}\dy} \int_\Omega \frac{h e^{v_\mu}}{\int_\Omega h e^{v_\mu}\dy} v_\mu \dx.
\end{align}
For $\mu=0$ the unique solution is given by~$v_0=0$ and the energy is just
\begin{align}
 E_0\equiv E(v_0)=
 \frac{1}{2}\frac{\int_\Omega\int_\Omega h(x) G(x,y)h(y)\dx\dy}{\parenthesis{\int_\Omega h\dx}^2},
\end{align}
where $G(x,y)$ is the usual Green function with Dirichlet boundary conditions. Letting $\Gamma$ be the branch of solutions $(\mu,v_\mu)$ crossing the origin $(0,0)$, we have the following result.

\begin{thm} \label{thm2}
Let $(\Omega,\ba)$ be a domain of first kind and let $\Gamma$ be as above. There exists a global parametrization
$$
	(\mu_E,v_{\mu_E})\in\Gamma, \quad E(\mu_E)=E\in(0,+\infty),
$$
such that we have, see Figures \ref{pic:mu-E}, \ref{fig}:
\begin{enumerate}
	\item \emph{(Asymptotics):} $\mu_E\to-\infty$ for $E\to0^+$ and $\mu_E\to0^+$ for $E\to+\infty$.
	
	\medskip
	
	\item \emph{(Monotonicity):} there exists $E_*$ such that $\mu_E$ is strictly increasing for $E<E_*$ and strictly decreasing for $E>E_*$.
\end{enumerate}
\end{thm}

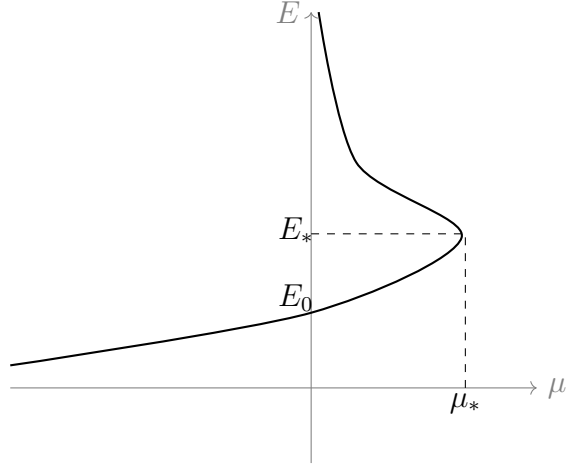
\begin{figure}[h]
 \centering
 \begin{tikzpicture}
 \draw[gray, thin,->] (-4,0) -- (3,0) node[right]{$\mu$};
 \draw[gray, thin,->] (0,-1) -- (0,5) node[left]{$E$};
 %\draw[dashed]  (-4,0.5) - - (3,0.5);
 \draw[thick] plot [smooth] coordinates{(-4,0.3) (0,1) (2,2) (0.6,3) (0.1,5)};
 %\draw (-0.22, 0.24) node{$\underline{E}$};
 \draw (-0.2,1.2) node{$E_0$};
 \draw (-0.2,2.1) node{$E_*$};
 \draw[dashed] (0,2.05)--(2, 2.05);
 \draw[dashed] (2.05,0)--(2.05,2.05);
 \draw (2.05,-0.2) node{$\mu_*$};
\end{tikzpicture}\\
 \caption{The shape of the bifurcation diagram.}
 \label{pic:mu-E}
\end{figure}

\begin{rmk} \label{rem:unique}
Let us point out that we are considering $\Omega$ simply connected to ensure uniqueness of the solution for $\lambda<8\pi$, see \cite{bartolucci2009uniqueness}. However, under suitable assumptions on the weight function in the spirit of \cite{BL-ann}, it is straightforward to check that uniqueness still hold on multiply connected domains and thus Theorem \ref{thm2} covers this case as well.
\end{rmk}

\medskip

This seems to be the first qualitative result about the global bifurcation diagram of \eqref{eq:L-mu}. Observe that the shape of the branch of solutions resembles the well-known one of the model regular radial case on the disk, see for example \cite{suzuki1992global}, even for non-symmetric configurations $(\Omega,\ba)$.

\medskip

By exploiting Theorem \ref{thm1} jointly with some further spectral analysis it is possible in particular to catch the monotonic behavior of the quantities describing the diagram. We finally point out that we study the linearization of an essentially constrained problem \eqref{eq:MP-lambda} and we thus cannot rely on standard techniques based on the simpleness of the first eigenvalue, positivity of the first eigenfunction or on the maximum principle.

\

The paper is organized as follows. In section \ref{sec:prelim} we collect some useful results and introduce the spectral setting, in section \ref{sec:energy} we study the energy of the solutions and in section \ref{sec:diagram} we finally describe the bifurcation diagram.

\

\section{Preliminaries and notations} \label{sec:prelim}
Here we introduce the spectral theory and list some of its properties. Given~$\lambda\in\R$ and~$\psi\in H^1_0(\Omega)$, we write
\begin{align}
 \RH_\lambda(\psi)\equiv
 \frac{h e^{\lambda\psi}}{\int_{\Omega} h e^{\lambda\psi}\dx}.
\end{align}
Then~$\RH_\lambda(\psi)\dx$ is a probability measure on~$\Omega$, 
and we can decompose a Sobolev function~$u\in H^1(\Omega)$ into the average part and oscillation part:
\begin{align}
 u=\average{u}+ \oscil{u}, & &
 \mbox{ where } \quad
 \average{u} =\int_\Omega u \RH_\lambda \dx , \quad
 \oscil{u}= u-\average{u}.
\end{align}
Then we know that the Poincar\'e constant
\begin{align}
 C_P(\lambda,\Omega)
 \coloneqq \inf\braces{\frac{\int_{\Omega}|\nabla\phi|^2\dx}{\int_\Omega \phi^2 \RH_\lambda\dx} \mid \phi\in H^1(\Omega), \average{\phi}=0 }
\end{align}
is positive.

\

For a fixed~$\alpha\in (0,1)$ consider the operator
\begin{align}
 F\colon \R\times C^{2,\alpha}_0(\bar{\Omega})\to C^\alpha(\bar{\Omega}), \qquad
 F(\lambda,\psi)=-\Delta \psi-\RH_\lambda(\psi).
\end{align}
The linearization of~$F$ with respect to the~$\psi$ variable is
\begin{align}
 L_{(\lambda,\psi)}\eta
 \equiv D_\psi F(\lambda,\psi)\eta
 =-\Delta\eta-\lambda\RH_\lambda \oscil{\eta}.
\end{align}
It defines an elliptic operator which is also formally self-adjoint on the subspace of $L^{2}(\Omega)$ functions of vanishing mean ($\average{\phi}=0$), see~\cite{bartolucci2019global} for further details. At this point we set up a suitable spectral theory which will turn out to be natural for our problem.
\begin{dfn}
$\sigma\in\R$ is an eigenvalue of the operator $ L_{(\lambda,\psi)}$ if there exists a non-trivial $\phi\in H^1_0(\Omega)$ such that
$$
 L_{(\lambda,\psi)}\phi=\sigma \RH_\lambda\oscil{\phi}.
$$
\end{dfn}
The definition differs from the standard one by requiring zero (weighted) average eigenfunctions in the above right-hand side. However, it is possible to check that one can built on it a spectral theory with the usual properties and we refer to \cite{bartolucci2019global} for further comments in this respect.

\medskip

The linearized operator can be used to analyze the set of solutions, as the following lemma indicates:
\begin{lemma}
 Let~$\lambda_0\in\R$ and~$\psi_{\lambda_0}$ be a solution of~$(MP_{\lambda_0})$.
 If~$0$ is not an eigenvalue of~$L_{(\lambda_0,\psi_{\lambda_0})}$, then
 \begin{enumerate}
  \item[(i)] $L_{(\lambda_0,\psi_{\lambda_0})}\coloneqq C^{2,\alpha}_0(\Omega) \to C^\alpha(\Omega)$ is an isomorphism.
  \item[(ii)] There exists an open neighborhood~$\mathscr{U}=I\times B\subset \R\times C^{2,\alpha}_0(\Omega)$ of~$(\lambda_0,\psi_{\lambda_0})$ such that the set of solutions of~\eqref{eq:MP-lambda} in~$\mathscr{U}$ forms an analytic curve
  \begin{align}
   I\to B,\quad \lambda \to \psi_\lambda
  \end{align}
  which passes through~$(\lambda_0,\psi_{\lambda_0})$.
\end{enumerate}
\end{lemma}

\begin{proof}
 The proof follows standard arguments based on Fredholm alternative and implicit function theorem and we refer to~\cite{bartolucci2019global} for details.
\end{proof}

\

Let~$(\lambda,\psi_\lambda)$ be a solution of~\eqref{eq:MP-lambda}.
To see whether~$L_{(\lambda,\psi_{\lambda})}$ is invertible or not, we consider
\begin{align}
 \sigma_1(\lambda,\psi_\lambda)
 \coloneqq \inf_{\phi\in H^1_0(\Omega)\setminus\{0\}}
 \frac{\int_\Omega |\nabla\phi|^2\dx -\lambda\int_\Omega \oscil{\phi}^2\; \RH_\lambda\dx}{\int_\Omega \oscil{\phi}^2 \;\RH_\lambda\dx}.
\end{align}
We have the following crucial observation
\begin{lemma}
For any $\lambda<8\pi$ we have $ \sigma_1(\lambda,\psi_\lambda)>0$.
\end{lemma}
\begin{proof}
Consider the standard eigenvalue
\begin{align}
 \tau_1(\lambda,\psi_\lambda)
 \coloneqq \inf_{\phi\in H^1_0(\Omega)\setminus\{0\}}
 \frac{\int_\Omega |\nabla\phi|^2\dx -\lambda\int_\Omega \oscil{\phi}^2\; \RH_\lambda\dx}{\int_\Omega \phi^2 \;\RH_\lambda\dx}
\end{align}
for which we know that $\tau_1(\lambda,\psi_\lambda)>0$ for any $\lambda<8\pi$, see \cite{bartolucci2009uniqueness}. The thesis then follows just by observing that
$$
	\sigma_1(\lambda,\psi_\lambda)\geq \tau_1(\lambda,\psi_\lambda)>0.
$$
\end{proof}

\medskip

Hence for any solution~$(\lambda,\psi_\lambda)$, the linearization operator~$L_{(\lambda,\psi_\lambda)}$ has no zero eigenvalue, hence is invertible, and the solutions~$\psi_\lambda$ have analytic dependence on~$\lambda$.

\

\section{The energy of solutions} \label{sec:energy}

The aim of this section is to study the energy along the branch of solutions~$(\lambda,\psi_\lambda)$,
\begin{align}
 E_\lambda\coloneqq \frac{1}{2}\int_{\Omega}|\nabla\psi_\lambda|^2 \dx,
\end{align}
showing its monotonicity and describing its asymptotics. Observe that the energy is nothing but the~$\RH_\lambda$-average of the solution~$\psi_\lambda$ (up to a constant~$1/2$):
\begin{align}
 E_\lambda
 = \frac{1}{2}\int_{\Omega}|\nabla\psi_\lambda|^2 \dx
 = \frac{1}{2}\int_\Omega (-\Delta\psi_\lambda)\psi_\lambda\dx
 = \frac{1}{2}\int_{\Omega} \RH_\lambda \psi_\lambda \dx
 = \frac{1}{2}\average{\psi_\lambda}.
\end{align}
We divide the proof of Theorem \ref{thm1} in the next subsections.

\

\subsection{Monotonicity of the energy}
We start by showing the monotonic property; we have to exploit here the modified spectral theory introduced in the previous section. Let
\begin{align}
 \eta_\lambda\equiv \frac{\p\psi_\lambda}{\p\lambda},
\end{align}
which satisfies the equation
\begin{align}\label{eq:L-eta}
 L_{(\lambda,\psi_\lambda)}\eta_\lambda
 =-\Delta\eta_\lambda-\lambda\RH_\lambda\oscil{\eta_\lambda}
 =\RH_{\lambda} \cdot \oscil{\psi_\lambda}.
\end{align}
 Observe that
\begin{align}
 \frac{\dd E_\lambda}{\dd \lambda}
 =& \frac{1}{2}\frac{\dd}{\dd\lambda}\int_{\Omega} |\nabla\psi_\lambda|^2\dx
   =\int_{\Omega} (-\Delta\eta_\lambda) \psi_\lambda \dx
\end{align}
and using~\eqref{eq:L-eta} we get
\begin{align}\label{eq:E-lambda-1}
 \frac{\dd E_\lambda}{\dd \lambda}
 =\int_{\Omega} \lambda\RH_\lambda \oscil{\eta_\lambda}\cdot \psi_\lambda \dx + \int_\Omega \RH_\lambda \oscil{\psi_\lambda}\cdot \psi_\lambda\dx.
\end{align}
At this stage, we need to use the (weighted) eigenfunctions~$\{\phi_j\}$ of~$L_{(\lambda,\psi_\lambda)}$ as an orthonormal basis for~$L^2(\RH_\lambda\dx)$:
\begin{align}
 L_{(\lambda,\psi_\lambda)}\phi_j=\sigma_j \RH_\lambda\oscil{\phi_j}, & &
 \int_\Omega \oscil{\phi_j} \oscil{\phi_k}\;\RH_\lambda\dx=\delta_{jk}, & &
 \phi_j\in H^1_0(\Omega).
\end{align}
Here~$\sigma_j\equiv \sigma_j(\lambda,\psi_\lambda)$.
This is possible since the measure~$\RH_\lambda\dx$ only vanishes at a discrete set of points.
Note that~$\phi_0=const$ is the trivial eigenfunction, which is not considered here since we need~$\phi_j\in H_0^1(\Omega)$.

In terms of this basis, we can write
\begin{align}
 \oscil{\psi_\lambda}=\sum_{j=1}^\infty a_j\oscil{\phi_j},
\quad
 \oscil{\eta_\lambda}=\sum_{j=1}^\infty b_j\oscil{\phi_j}.
\end{align}
Testing~\eqref{eq:L-eta} against~$\oscil{\phi_j}$, we see that
\begin{align}
 \sigma_j b_j =a_j, \quad \forall j\ge 1.
\end{align}

Then, thanks to the orthonormal assumption on this basis,~\eqref{eq:E-lambda-1} implies
\begin{align}
 \frac{\dd E_\lambda}{\dd \lambda}
 =\int_{\Omega} \lambda\RH_\lambda \oscil{\eta_\lambda}\cdot \psi_\lambda \dx + \int_\Omega \RH_\lambda \oscil{\psi_\lambda}\cdot \psi_\lambda\dx
 =\sum_{j\ge 1} (\lambda a_j b_j + a_j^2)
 =\sum_{j\ge 1} (\lambda+\sigma_j)\sigma_j b_j^2 .
\end{align}
We have seen~$\sigma_j\ge \sigma_1>0$.
Meanwhile~$\sigma_j+\lambda\ge  C_P(\lambda,\Omega)>0$ is always  positive.
Note that~$\oscil{\psi_\lambda}\neq 0$ and hence~$\oscil{\eta_\lambda}\neq 0$, by~\eqref{eq:L-eta}.
Therefore
\begin{align}
 \frac{\dd E_\lambda}{\dd \lambda}>0,
\end{align}
i.e.~$E_\lambda$ is a strictly increasing function of~$\lambda$, for~$\lambda\in (-\infty,8\pi)$, as claimed.

\

\subsection{Asymptotics of the energy} We next discuss the asymptotic behavior of the energy; in particular, since it is monotone, it has limits as~$\lambda\to -\infty$ and as~$\lambda\to 8\pi^-$. We now consider the two cases separately.
\begin{lemma}\label{lem8pi}
It holds $\displaystyle{\lim_{\lambda\to8\pi^-}}E_\lambda=+\infty$.
\end{lemma}

\begin{proof}
 Argue by contradiction: if not, then
\begin{align}
 \lim_{\lambda\to8\pi^-}E_\lambda= \bar{E}<\infty.
\end{align}
Standard regularity theory implies that~$\|\psi_\lambda\|_{C^{2,\alpha}}$ are uniformly bounded, hence by Ascoli-Arzela, there is a limit~$\varphi\in C^{2,\alpha}_0(\bar{\Omega})$ solving~\eqref{eq:MP-lambda} with~$\lambda=8\pi$, contradicting the choice of the domain, which is assumed to be of first kind.
\end{proof}
\

We now consider the case $\lambda\to-\infty$. As already pointed out in the introduction we can not follow the argument in \cite{bartolucci2019ontheglobal} due to presence of the singular weight. We thus propose an ad hoc argument based on approximation and test functions. The approximation process is needed since we rely on the existence theory for free energy functionals (see for example \cite{csw}) which would require $h$ to be strictly positive. Let us fix some notation.

\

Consider the space of probability densities
\begin{align}
 \cP(\Omega) \coloneqq
 \braces{\rho \in L^1(\Omega)\mid \rho \ge 0\;  a.e., \;  \int_\Omega \rho \dx =1}.
\end{align}
Given~$\rho\in\cP(\Omega)$, its energy is
\begin{align}
 \cE(\rho) \coloneqq
 \frac{1}{2}\int_\Omega \rho(x) (G*\rho)(x)\dx,
\end{align}
where~$G*\rho$ is the convolution with~$\rho$:
\begin{align}
 (G*\rho)(x) = \int_\Omega G(x,y) \rho(y)\dy.
\end{align}
If we denote~$\varphi_\rho=G*\rho$, then~$\varphi_\rho$ is the unique solution of
\begin{align}
 \begin{cases}
  -\Delta \varphi_\rho = \rho, & \mbox{ in } \Omega, \\
  \varphi_\rho=0, & \mbox{ on } \p\Omega
 \end{cases}
\end{align}
and the energy of the density~$\rho$ is actually the Dirichlet energy of~$\varphi_\rho$:
\begin{align}
 E(\varphi_\rho)=\frac{1}{2}\int_\Omega |\nabla \varphi_\rho|^2\dx
 =\frac{1}{2}\int_{\Omega} (-\Delta \varphi_\rho) \varphi_\rho\dx
 =\frac{1}{2}\int_\Omega \rho (G*\rho)\dx
 =\cE(\rho).
\end{align}
The entropy of~$\rho \in \cP(\Omega)$ is given by
\begin{align}
 \cS(\rho)\coloneqq -\int_\Omega (\rho \log \rho)\dx.
\end{align}
%and the existence of solutions with prescribed energy is then handled through the classical Microcanonical %Variational Principle looking at solutions of the maximizing problem
%\begin{align}\label{eq:MVP} \tag{MVP}
% \sup\braces{\cS(\rho)\mid \rho \in \cP_E(\Omega)}
%\end{align}
%where~$\cP_E(\Omega)$ denotes the subset of probability measures at energy level~$E>0$:
%\begin{align}
% \cP_E(\Omega;h)\coloneqq \braces{\rho\in \cP(\Omega)\mid \cE(\rho)=E}.
%\end{align}
%The set~$\cP_E(\Omega;h)$ is nonempty for any~$E>0$.
%Indeed, let~$(A_t)_{t\in \R_+}$ be a family of sets, continuous in~$t$ in the sense that~$\mathds{1}_{A_t}$ are continuous in~$t$ with respect to the~$L^1$ norm.
%Assume that for~$t\searrow 0$, the sets~$A_t$ are shrinking disks with radii~$t$ centered at a regular point in~$\Omega$, while for~$t\nearrow +\infty$ the set~$A_t$ are the shrinking collars of~$\p\Omega$.
%The consider the measures~$\rho_t\equiv \frac{1}{|A_t|}\mathds{1}_{A_t}$.
%One can readily compute to see that their energies have range~$(0,+\infty)$.
%Note that this construction is independent of the function~$h$.
We now prove the following.
\begin{prop}\label{prop:lower limit of E}
 For any~$\eps>0$ there exists $\lambda\in\R$ and a solution~$\psi_\lambda$ of~\eqref{eq:MP-lambda} with energy~$E(\psi_\lambda)<\eps$. In particular,
\begin{align}
\lim_{\lambda \to-\infty} E_{\lambda}=0.
\end{align}
\end{prop}
\begin{proof}
We take~$h_n$ a sequence of approximating, smooth, positive weights such that~$h_n \to h$ in~$C^{\alpha}(\Omega)$, for some $\alpha\in(0,1)$. Fix a~$\lambda<0$ and consider the functional~$\cF_{\lambda,n}\colon \cP(\Omega)\to \R$,
 \begin{align}
  \cF_{\lambda,n}(\rho)\coloneqq
  \int_\Omega \rho(x)\log\rho(x)\dx
  -\frac{\lambda}{2}\int_\Omega \rho(x) (G*\rho)(x)\dx
  -\int_\Omega \rho(x)\log h_n(x)\dx.
 \end{align}
 Observe that for $\lambda<0$ the functional $\cF_{\lambda,n}$ is strictly convex and it is well known that  
 there exists a unique minimizer obtained by the following variational principle:
 \begin{align}
  \cF_{\lambda,n}(\rho_{\lambda,n}) = \min\braces{\cF_{\lambda,n}(\rho)\mid \rho\in \cP(\Omega)},
 \end{align}
see for example \cite{csw}, where a more general case is considered. 
 
 The critical point~$\rho_{\lambda,n}$ of~$\cF_{\lambda,n}$ satisfies the equation
 \begin{align}
  \log \rho_{\lambda,n} =\log h_n + \lambda G*\rho_{\lambda,n} + c-1,
 \end{align}
 where~$c$ is the Lagrange multiplier from the constraint~$\int_{\Omega} \rho=1$, see \cite{csw} for details.

 In terms of~$\psi_{\lambda,n}\coloneqq G*\rho_{\lambda,n}$, we have
 \begin{align}\label{eq:psi-lambda-n}
  \begin{cases}
	-\Delta \psi_{\lambda,n}=\rho_{\lambda,n}=\displaystyle{\frac{ h_n e^{\lambda \psi_{\lambda,n}}}{\int_\Omega h_n e^{\lambda \psi_{\lambda,n}}\dx}} & \mbox{ in } \Omega , \\
  \psi_{\lambda,n}=0 & \mbox{ on } \p\Omega.
	\end{cases}
 \end{align}

 Note that~$\int \limits_{\Omega}\rho_{\lambda,n}=1$, so that (\cite{St4}) for any~$p<2$,
 \begin{align}
  \|\psi_{\lambda,n}\|_{W^{1,p}}
  \le C(\Omega,p).
 \end{align}
 In particular, we can assume that, up to a subsequence,~$\psi_{\lambda,n}$ converges weakly in~$W^{1,p}_0(\Omega)$ and strongly in~$L^q(\Omega)$ for any~$q<\infty$, to a limit function~$\psi_{\lambda,\infty}\in W^{1,p}_0(\Omega)$.
 As another consequence, up to a subsequence again,
 \begin{align}
  \int_\Omega \psi_{\lambda,n} h_n\dx \to \int_\Omega \psi_{\lambda,\infty} h\dx.
 \end{align}
 Since~$\lambda<0$, by maximum principle,~$\psi_{\lambda, n}$ is non-negative, so is~$\psi_{\lambda,\infty}$.

\medskip

 To see that~$\psi_{\lambda,\infty}$ is a solution of the mean field equation, we need to estimate the right-hand side of~\eqref{eq:psi-lambda-n}.   The numerator is easily estimated since~$\lambda\psi_{\lambda,n}\le 0$ and~$h_n\to h$ in~$C^{\alpha}$.
 For the denominator, we have
 \begin{align}
  \int_\Omega h_n e^{\lambda \psi_{\lambda,n}} \dx
  =& \| h_n\|_{L^1} \int_\Omega e^{\lambda \psi_{\lambda,n}} \frac{ h_n \dx}{\|h_n\|_{L^1}} \\
  \ge& \|h_n\|_{L^1} \exp\parenthesis{\lambda\int_\Omega \psi_{\lambda,n} \frac{h_n\dx}{\|h_n\|_{L^1}}} \\
  \to& \|h\|_{L^1} \exp\parenthesis{\frac{\lambda}{\|h\|_{L^1}}\int_\Omega \psi_{\lambda,\infty} h\dx }
 \end{align}
 as~$n\to\infty$, where the Jensen inequality is applied to the second line.
 The limit is non-zero, hence we get a uniform lower bound for the denominators in the right-hand side of~\eqref{eq:psi-lambda-n}, hence also a uniform upper bound of the whole right-hand side.
 Then by bootstrap argument we see that the sequence~$\psi_{\lambda,n}$ converges in~$C^{2,\alpha}$ for any~$\alpha$ sufficiently small.
 Therefore, the function~$\psi_{\lambda,\infty}$ satisfies the equation~\eqref{eq:MP-lambda}:

 \begin{align}
\begin{cases}
  -\Delta \psi_{\lambda,\infty}=\dfrac{h e^{\lambda \psi_{\lambda,\infty}}}{\int_\Omega h e^{\lambda\psi_{\lambda,\infty}} \dx } & \mbox{ in } \Omega, \\
  \psi_{\lambda,\infty}=0 & \mbox{ on } \p\Omega.
\end{cases}	
 \end{align}

 Since~$\lambda<0$ this equation admits only one solution so~$\psi_{\lambda,\infty}=\psi_\lambda$. Uniqueness 
 here can be deduced by observing that in fact via a dual formulation $\psi_{\lambda}$ is a solution of \eqref{eq:MP-lambda} if and only if it is a critical point of the convex (for $\lambda<0$) functional 
 $\frac12\int\limits_{\Omega} |\nabla \psi|^2-\frac{1}{\lambda}\log\left(\int\limits_{\Omega} h e^{\lambda\psi}\right)$ on $H^1_0(\Omega)$.

\

At last it remains to show that~$\psi_{\lambda}$ can assume very low energy levels, for which we will use a comparison argument.
 By the strong convergence in~$C^2$, we know that
 \begin{align}
  E(\psi_\lambda)=\lim_{n\to\infty} E(\psi_{\lambda,n}).
 \end{align}
 Recalling the notation introduced right after Lemma \ref{lem8pi}, it is then sufficient to estimate the energy of~$\psi_{\lambda,n}$:
 \begin{align}
  E(\psi_{\lambda,n})
  =\cE(\rho_{\lambda,n})
  =\frac{1}{2}\int_\Omega \rho_{\lambda,n} G*\rho_{\lambda,n}\dx
  =\frac{-1}{\lambda} \cF_{\lambda,n}(\rho_{\lambda,n})
  -\frac{1}{\lambda}\cS(\rho_{\lambda,n})
  -\frac{1}{\lambda}\int_\Omega \rho_{\lambda,n} \log h_n \dx.
 \end{align}
 By the Jensen inequality we know that the entropy is non-positive,
 \begin{align}
  \cS(\rho_{\lambda,n})\le 0.
 \end{align}
 For the linear part we have
 \begin{align}
  \int_{\Omega}\rho_{\lambda,n} \log h_n\dx
  \le \log(1+\|h\|_{C^0}) \int_{\Omega} \rho_{\lambda,n}\dx
  = \log(1+\|h\|_{C^0}).
 \end{align}
 Thus, recalling also $\lambda<0$,
 \begin{align}
  \cE(\rho_{\lambda,n})\le \frac{1}{|\lambda|} \cF(\rho_{\lambda,n}) + \frac{1}{|\lambda|} \log(1+\|h\|_{C^0}).
 \end{align}
  On the other hand, it is known that there exist probability measures with arbitrarily small energy. For example one can consider~$\rho_{\delta}= |C_\delta|^{-1} \mathds{1}_{C_\delta}$ where~$C_\delta\coloneqq\braces{x\in\Omega\mid \dist(x,\p\Omega)<\delta}$ denotes the inner~$\delta$-collar of the boundary. It is well-known that
$$
	\cE(\rho_\delta)\to0
$$
for~$\delta\to 0$. Therefore,
 \begin{align}
  \frac{1}{|\lambda|}\cF(\rho_{\lambda,n})\le \frac{1}{|\lambda|}\cF(\rho_\delta)
  = \cE(\rho_\delta){+\frac{1}{\lambda}\cS(\rho_\delta)}+\frac{1}{\lambda} \int_\Omega \rho_\delta \log h_n\dx.
 \end{align}
 For a fixed small $\delta$, sending~$\lambda\to -\infty$, we see that~$\cE(\rho_{\lambda,n})$ can be arbitrarily small, as claimed.
\end{proof}

\

This completes the proof of Theorem \ref{thm1} and  we get a picture for the solutions in the~$(\lambda, E)$-plane as in Figure~\ref{pic:lambda-E}.

\

\section{The bifurcation diagram} \label{sec:diagram}
In this section we provide the proof of Theorem \ref{thm2}, analyzing the branch of solutions of
\begin{align}\tag{$LP_\mu$} \label{eq:LP-mu}
 \begin{cases}
  -\Delta v= \mu h e^v & \mbox{ in } \Omega \\
  v=0 & \mbox{ on } \p\Omega .
 \end{cases}
\end{align}
Since~$\lambda\mapsto E_\lambda$ is a monotone bijection, and~$\mu=\mu_\lambda$ is explicit, we first consider the graph of
\begin{align}
 E\mapsto \mu_{\lambda(E)}.
\end{align}
Note that
\begin{align}
 \frac{\dd\mu}{\dd E}
 =& \frac{\dd\mu}{\dd\lambda}\frac{\dd \lambda}{\dd E}
   =\frac{\dd\mu_\lambda}{\dd \lambda}\frac{1}{\frac{\dd E_\lambda}{\dd\lambda}}.
\end{align}
Since~$\frac{\dd E_\lambda}{\dd \lambda}>0$, we need to analyze the sign of~$\frac{\dd\mu_\lambda}{\dd\lambda}$.

\

\subsection{The sign of \texorpdfstring{$\p_\lambda \mu$}{Jacobian}}

From~\eqref{eq:from lambda to mu} it follows that
\begin{align}
 \frac{\dd\mu_\lambda}{\dd\lambda}
 =& \frac{\dd}{\dd\lambda}\parenthesis{\frac{\lambda}{\int_{\Omega} h e^{\lambda\psi_\lambda}\dx} }
 =\frac{\dd}{\dd\lambda}\parenthesis{\frac{\lambda}{\int_{\Omega} h e^{u_\lambda}\dx} }
 =\frac{1}{\int_{\Omega} h e^{u_\lambda}\dx}
   \parenthesis{1-\lambda\int_{\Omega}\RH_\lambda\frac{\p u_\lambda}{\p\lambda}\dx }.
\end{align}
For later convenience we denote
\begin{align}
 z_\lambda\equiv u'_\lambda=\frac{\p u_\lambda}{\p\lambda}.
\end{align}
Then
\begin{align}\label{eq:mu-lambda and g}
 \frac{\dd\mu_\lambda}{\dd\lambda}
 =\frac{1}{\int_{\Omega} h e^{u_\lambda}\dx} (1-\lambda \average{z_\lambda}).
\end{align}
We thus study the sign of
\begin{align}
 g(\lambda)\coloneqq 1-\lambda\average{z_\lambda}.
\end{align}
We follow here the ideas from~\cite{bartolucci2019ontheglobal}.

\

{\textbf{Step 1.}} We claim $g(\lambda)>0$ for $\lambda\in(-\infty,\delta)$ for some $\delta>0$.

\medskip

We first prove $\average{z_\lambda}>0$. Indeed, the equation for~$z_\lambda$ is
        \begin{align}
         -\Delta z_\lambda =\RH_\lambda+\lambda\RH_\lambda z_\lambda-\lambda\RH_\lambda\average{z_\lambda},
       \end{align}
       which is obtained from~$-\Delta u_\lambda=\mu_\lambda h e^{u_\lambda}=\lambda \RH_\lambda$. Now,  using~$z_\lambda=\average{z_\lambda}+\oscil{z_\lambda}$, we get
       \begin{align}
         \int_\Omega |\nabla z_\lambda|^2\dx
        =\average{z_\lambda}+\lambda\average{\oscil{z_\lambda}^2}.
       \end{align}
       Hence
       \begin{align}
        \average{z_\lambda}=\int_\Omega|\nabla z_\lambda|^2\dx -\lambda\int_\Omega\RH_\lambda\oscil{z_\lambda}^2\dx
        \ge \sigma_1(\lambda;\Omega)\int_\Omega\RH_\lambda\oscil{z_\lambda}^2\dx
       \end{align}
       which is positive since~$\sigma_1>0$, for any~$\lambda\in(-\infty, 8\pi)$.

\medskip			

Now, by the definition of $g$ we readily have $g(\lambda)>0$ for $\lambda\in(-\infty,0)$. Moreover, $g(0)=1$ and
since~$g(\lambda)$ is continuous,~$g(\lambda)>0$ for~$\lambda>0$ small. This proves the claim.

\

{\textbf{Step 2.}} We prove now
$$
     \lim_{\lambda\to8\pi^-} g(\lambda)=-\infty.
$$

\medskip

Indeed, recalling $\eta_\lambda\equiv \frac{\p\psi_\lambda}{\p\lambda}$ and using the correspondence~$u_\lambda=\lambda\psi_\lambda$, we get
            \begin{align}
             z_\lambda=(u_\lambda)' = (\lambda\psi_\lambda)'=\psi_\lambda+ \lambda\psi_\lambda'
            \end{align}
            and consequently
            \begin{align}
             \average{z_\lambda}
             =\average{\psi_\lambda}
              +\lambda\average{\eta_\lambda}
             =2E_\lambda+\lambda E_\lambda' \to +\infty,
            \end{align}
            as $\lambda\to 8\pi^-$, where we have used
            \begin{align}
             E_\lambda
             =& \frac{1}{2}\int_\Omega \RH_\lambda\psi_\lambda \dx
               =\average{\psi_\lambda}, \\
             E'_\lambda
             =& \int_\Omega (-\Delta \eta_\lambda)\psi_\lambda\dx
               =\int_\Omega (-\Delta\psi_\lambda)\eta_\lambda\dx
               =\int_\Omega \RH_\lambda\eta_\lambda\dx
               =\average{\eta_\lambda}.
            \end{align}

\

{\textbf{Step 3.}} We claim that $g|_{(0,8\pi)}$ has only one zero.

\medskip

Clearly, by Steps 1, 2 and since~$g(\lambda)$ is continuous, there exists a~$\lambda_* \in (0, 8\pi)$ such that~$g(\lambda_*)=0$. We now prove this is the only zero. The argument works as in \cite{bartolucci2019ontheglobal} so we omit the details and briefly sketch the main steps. Suppose with no loss of generality that $\lambda_*>0$ is the smallest zero. It suffices to show that $g|_{(\lambda_*, 8\pi)} <0$. If not, then it is possible to show that there would exist
            \begin{align}
             \hat{\lambda}\coloneqq \sup\braces{\tau>\lambda_*\mid g|_{(\lambda_*,\tau)}\le 0}
            \end{align}
with $g(\hat{\lambda})=0$ and~$g'(\hat{\lambda})\ge 0$. 		

\medskip

Now, we can put the equation for~$g(\lambda)$ into the form
       \begin{align}
        g'(\lambda)=a(\lambda) g(\lambda)+b(\lambda)
       \end{align}
with $b(\lambda)=-\lambda^2\average{z_\lambda^3}$ and for some $a(\lambda)$, see \cite[Lemma 5.2]{bartolucci2019ontheglobal}. Therefore,
            \begin{align}
             0\leq g'(\hat{\lambda})
             =a(\hat{\lambda}) \underbrace{g(\hat{\lambda})}_{=0}+ b(\hat{\lambda})
             =b(\hat{\lambda})
             =-\hat{\lambda}^2\left<z_{\hat{\lambda}}^3 \right>_{\hat{\lambda}}.
            \end{align}
On the other hand, we can use Proposition 5.1 (iii) in \cite{bartolucci2019ontheglobal} to conclude that, since $g(\hat{\lambda})=0$, then~$z_{\hat{\lambda}}\geq0$ in~$\Omega$ and $\langle z_{\hat{\lambda}}^3 \rangle_{\hat{\lambda}}> 0$, which yields to a contradiction.

\
						
This concludes the study of $g(\lambda)$ and we depict its graph in Figure \ref{fig:g}.

\begin{figure}[h]
\centering
 \begin{tikzpicture}
 \draw[gray, very thin,->] (0,-3) -- (0,3) node[left]{$g$};
 \draw[gray, very thin,->] (-2,0) -- (4,0) node[right]{$\lambda$};
 \draw[dashed] (3,-3) -- (3,3);
 \draw (3.25, -0.20) node{$8\pi$};
 \draw[thick] plot [smooth] coordinates{(-2,0.5) (-1,0.95) (0,1.1) (1,0.96) (2.3,0) (2.9, -3)};
 \draw (0,1.5)  node[left]{$1$};
 \draw (1.9,0.2) node{$\lambda_*$};
\end{tikzpicture}\\
 \caption{The function~$g(\lambda)$.}
\label{fig:g}
\end{figure}
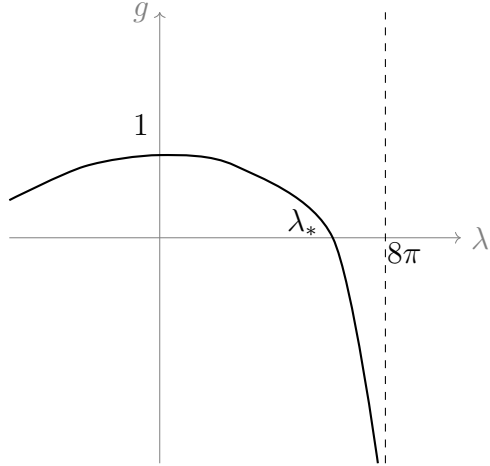

\

By~\eqref{eq:mu-lambda and g} we finally have what we were looking for:
\begin{align}
 \frac{\dd\mu_\lambda}{\dd\lambda}
 \begin{cases}
  >0, & \mbox{ if } \lambda\in (-\infty, \lambda_*) \\
  =0, & \mbox{ if } \lambda=\lambda_*\\
  <0, & \mbox{ if } \lambda\in (\lambda_*, 8\pi).
 \end{cases}
\end{align}

\

\subsection{The bifurcation diagram}
Let~$E_*\equiv E_{\lambda_*}$ with~$\lambda_*$ as above.
In the~$(E,\mu)$-plane, the solution is indicated by the curve
\begin{align}
 E\mapsto \lambda_E \mapsto \mu_{\lambda_E}\equiv \mu(E),
\end{align}
for~$E\in (0, +\infty)$.
We want to see the shape of this curve.
Recall that~$\frac{\dd E_\lambda}{\dd\lambda}>0$, and
\begin{align}
 \frac{\dd \mu(E)}{\dd E}
 =\frac{\dd\mu_\lambda}{\dd\lambda}
  \frac{1}{\frac{\dd E_\lambda}{\dd\lambda}}
 \begin{cases}
  >0, & \mbox{ if } \lambda\in (-\infty, \lambda_*) \quad  \Longleftrightarrow E\in (0, E_*) \\
  =0, & \mbox{ if } \lambda=\lambda_* \qquad \qquad \Longleftrightarrow E=E_*\\
  <0, & \mbox{ if } \lambda\in (\lambda_*, 8\pi) \qquad \Longleftrightarrow E\in (E_*, +\infty).
 \end{cases}
\end{align}
Also,~$\mu(E)=0$ iff~$\mu_{\lambda(E)}=0$ iff~$\lambda(E)=0$ iff~$E=E_0$.
Clearly~$E_*>E_0$ because~$\lambda_*>0$.
Thus we know that the curve~$(E,\mu(E))$ passes through the points
\begin{align}
 (E_0, 0) , \quad (E_*, \mu_{\lambda_*}\equiv \mu_*).
\end{align}

\medskip

As for the asymptotics, we know from monotonicity that the following limits exist
\begin{align}
 \lim_{E\to 0^+} \mu(E)= \lim_{\lambda\to -\infty} \mu_\lambda\equiv \mu_0, \quad
 \lim_{E\to+\infty} \mu(E)
 =\lim_{\lambda\to8\pi^-}\mu_\lambda\equiv \mu_1.
\end{align}

Note that for any~$\mu\leq0$, the equation~\eqref{eq:LP-mu} has a unique solution. Therefore,~$\mu_1\ge 0$ and~$\mu_0=-\infty$. Finally, by the definition of $\mu_\lambda$ in \eqref{eq:from lambda to mu} and the properties of the one-point blowup solutions as $\lambda\to8\pi^-$, we have $\mu_1=0$, see for example \cite{BL},

\medskip

Combining the above facts, we see that the solution curve in the~$(E,\mu)$-plane has the shape depicted in Figure \ref{fig}. By a reflection with respect to the diagonal line we get the desired solution curve in the~$(\mu,E)$-plane, see Figure~\ref{pic:mu-E}. This completes the proof of Theorem \ref{thm2}.

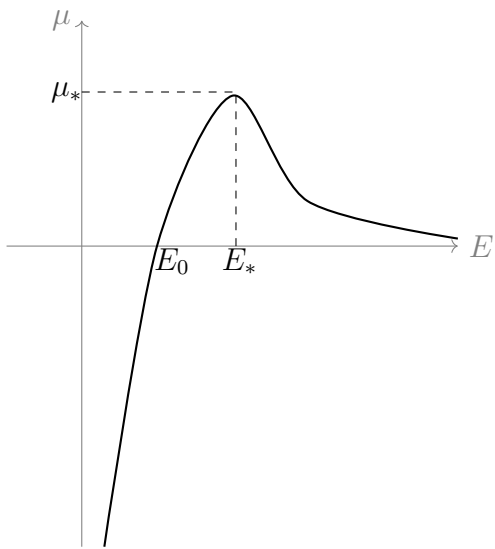
\begin{figure}[h]
\centering
 \begin{tikzpicture}
 \draw[gray, thin,->] (0,-4) -- (0,3) node[left]{$\mu$};
 \draw[gray, thin,->] (-1,0) -- (5,0) node[right]{$E$};
 %\draw[dashed]  (0.5,-4) - - (0.5, 3);
 \draw[thick] plot [smooth] coordinates{(0.3,-4) (1,0) (2,2) (3,0.6) (5,0.1)};
 %\draw (0.4,-0.22) node{$\underline{E}$};
 \draw (1.2,-0.2) node{$E_0$};
 \draw (2.1,-0.2) node{$E_*$};
 \draw[dashed] (2.05,0)--(2.05,2);
 \draw[dashed] (0,2.05)--(2.05,2.05);
 \draw (-0.2,2.05) node{$\mu_*$};
\end{tikzpicture}\\
 \caption{The shape of the bifurcation diagram.}
\label{fig}
\end{figure}

\

\

\end{document}